\documentclass[12pt, leqno, equal margin]{article} \textwidth=14cm
\usepackage[english]{babel} % francisation
\usepackage[T1]{fontenc} % pour taper les lettres accentuÃÆÃÆÃâÃÂ©es
\usepackage[latin1]{inputenc}
\usepackage[all]{xy}
\usepackage[dvips ,lmargin=2.5cm ,rmargin=2.5cm ,tmargin=2.5cm ,   bmargin=2.5cm]{geometry}

\usepackage{amsmath,amsthm}
\usepackage{amssymb}

\DeclareMathOperator{\pd}{pd} \DeclareMathOperator{\id}{id}
\DeclareMathOperator{\Ext}{Ext} \DeclareMathOperator{\Hom}{Hom}
\DeclareMathOperator{\fd}{fd} 
\DeclareMathOperator{\depth}{depth} \DeclareMathOperator{\Mod}{Mod}

\DeclareMathOperator{\wdim}{wgl-dim}
\DeclareMathOperator{\gdim}{l-gl-dim}

 \DeclareMathOperator{\Tor}{Tor}

\DeclareMathOperator{\FPD}{FPD}  \DeclareMathOperator{\fp}{fp}
\DeclareMathOperator{\FFD}{FFD}
\DeclareMathOperator{\lr}{\longrightarrow}
\DeclareMathOperator{\E}{E}

\DeclareMathOperator{\lMod}{lMod}
\usepackage{enumerate}

\usepackage{graphicx}

\newtheorem{thm}{Theorem}[section]
\newtheorem{cor}[thm]{Corollary}
\newtheorem{lem}[thm]{Lemma}
\newtheorem{prop}[thm]{Proposition}

\theoremstyle{definition}
\newtheorem{defin}[thm]{Definition}

\begin{document}

	%% Classification and key words; note that the 2010 classification is used:
	
	\renewcommand{\thefootnote}{\arabic{footnote}}
	\setcounter{footnote}{0}
	
	%%%%%%%%%%%%%%%%
	\title{On flatness and coherence with respect to modules of flat dimension at most one}

	\author{ Samir Bouchiba\footnote {Corresponding author: s.bouchiba@fs.umi.ac.ma} {} and Mouhssine El-Arabi\footnote {elarabimouh@gmail.com}\\
		{\footnotesize Department of Mathematics, Faculty of Sciences,}
		{\footnotesize University Moulay Ismail,}\\ {\small Meknes,
			Morocco} 
		}

	\maketitle

\begin{abstract}
	This paper introduces and studies homological properties of new classes of modules, namely, the $\mathcal F_1$-flat modules and the $\mathcal F_1^{\fp}$-flat modules, where $\mathcal F_1$ stands for the class of right modules of flat dimension at most one and $\mathcal F_1^{\fp}$ its subclass consisting of finitely presented elements. This leads us to introduce a new class of rings that we term $\mathcal F_1^{\fp}$-coherent rings as they behave nicely with respect to $\mathcal F_1^{\fp}$-flat modules as do coherent rings with respect to flat modules. The new class of $\mathcal F_1^{\fp}$-coherent rings turns out to be a large one and it includes coherent rings, perfect rings, semi-hereditary rings and all rings $R$ such that $\lim\limits_{\lr}\mathcal P_1=\mathcal F_1$. As a particular case of rings satisfying $\lim\limits_{\lr}\mathcal P_1=\mathcal F_1$ figures the important class of integral domains.
	\end{abstract}

\section{Introduction}
Throughout this paper, $R$ denotes an associative ring with unit element and the $R$-modules are supposed to be
unital. Given an $R$-module $M$, $M^{+}$ denotes the character $R$-module of $M$, that is, $M^{+}:=\Hom_{\mathbb{Z}}\Big (M, \dfrac{\mathbb{Q}}{\mathbb{Z}}\Big )$, $\pd_R(M)$ denotes the projective dimension of $M$,
$\id_R(M)$ the injective dimension of $M$ and $\fd_R(M)$ the flat dimension
of $M$. As for the global dimensions, $\gdim(R)$ designates the left global dimension of $R$ and $\wdim(R)$ the
weak global dimension of $R$. $\Mod(R)$ stands for the class of all right $R$-modules, $\mathcal P(R)$ stands for the class of all projective right $R$-modules, $\mathcal {I}(R)$ the class of all injective right
$R$-modules and $\mathcal F(R)$ the class of flat right modules. Also, we denote by $\mathcal F_1$ (resp., $\mathcal P_1$) the class of right $R$-modules $M$ such that $\fd_R(M)\leq 1$ (resp., $\pd_R(M)\leq 1$) and by $\mathcal F_1^{\fp}$ (resp., $\mathcal P_1^{\fp}$) the subclass of $\mathcal F_1$ (resp., $\mathcal P_1$) consisting of right $R$-modules which are finitely presented. Any unreferenced material is standard as in \cite{R,T,W,X}.

Our main purpose in this paper is to study the homological properties of the Tor-orthogonal classes $\mathcal F_1^{\top}$ and $\mathcal F_1^{\fp\top}$ of $\mathcal F_1$ and $\mathcal F_1^{\fp}$ that we term the class of $\mathcal F_1$-flat modules and the class of $\mathcal F_1^{\fp}$-flat modules, respectively, over an arbitrary ring $R$. Let us denote these two classes by $\mathcal F_1\mathcal F(R):=\mathcal F_1^{\top}$ and $\mathcal F_1^{\fp}\mathcal F(R):=\mathcal F_1^{\fp\top}$. Note that in the context of integral domains, by \cite[Lemma 2.3]{L}, $\mathcal F_1\mathcal F(R)$ coincides with the class of torsion-free $R$-modules. In other words, the $\mathcal F_1$-flat module notion extends that of the torsion-free module from integral domains to arbitrary rings, and thus our study will permit, in particular, to shed light on homological properties of torsion-free modules in the case when $R$ is a domain. 
In section 2, we introduce and study the $\mathcal F_1$-flat modules and $\mathcal F_1^{\fp}$-flat modules.  We mainly seek conditions on rings $R$ for which the two notions of $\mathcal F_1$-flat module and $\mathcal F_1^{\fp}$-flat module collapse. In particular, we prove that if $\lim\limits_{\lr}\mathcal P_1=\mathcal F_1$, then $\lim\limits_{\lr}\mathcal F_1^{\fp}=\mathcal F_1$ and thus $\mathcal F_1\mathcal F(R)=\mathcal F_1^{\fp}\mathcal F(R)$. This leads us to seek suffisant conditions on a ring $R$ for which $\lim\limits_{\lr}\mathcal F_1^{\fp}=\mathcal F_1$ or more generally $\lim\limits_{\lr}\mathcal P_1=\mathcal F_1$. It is worthwhile pointing out that if $R$ is an integral, then $\lim\limits_{\lr}\mathcal P_1=\mathcal F_1$ \cite[Theorem 3.5]{HT} and that this theorem is subsequently generalized by Bazzoni and Herbera in \cite[Theorem 6.7 and Corollary 6.8]{BH}.
Moreover, we exhibit an example of a ring $R$ such that $\lim\limits_{\lr}\mathcal F_1^{\fp}\neq\mathcal F_1$ showing that, in general, the concerned classes $\mathcal F_1\mathcal F(R)$ and $\mathcal F_1^{\fp}\mathcal F(R)$ are different.

Section 3 introduces the $\mathcal F_1^{\fp}$-coherent rings. This new class of rings behaves nicely with respect to $\mathcal F_1^{\fp}$-flat modules as do coherent rings with respect to flat modules. In particular, it is proven that a ring $R$ is $\mathcal F_1^{\fp}$-coherent if and only if any product of $\mathcal F_1^{\fp}$-flat modules is $\mathcal F_1^{\fp}$-flat if and only if $\lim\limits_{\lr}\mathcal P_1^{\fp}=\lim\limits_{\lr}\mathcal F_1^{\fp}$. It turns out that the class of $\mathcal F_1^{\fp}$-coherent rings is a large one and it includes integral domains, coherent rings, semi-hereditary rings and perfect rings which permits to unify all these classes of rings into one. Also, the class of $\mathcal F_1^{\fp}$-coherent rings includes all rings such that $\lim\limits_{\lr}\mathcal P_1=\mathcal F_1$.   Finally, we characterize rings $R$ for which all $R$-modules are $\mathcal F_1^{\fp}$-flat and discuss the homological dimensions of the $R$-modules as well as the global dimensions of $R$ in terms of the homological dimensions of the $\mathcal F_1$-flat modules.

	%%%%%%%%%%%%%%%%%%%%%%%%%%%%%%%%%%%%%%%%%%%%%%%%%%%%%%%%%%%%%%%%%%%%%%%%%%%%%%%%%%%%%%%%%%%%%%%%%%%%%%%%%%%%%%%%%%%%%%%%%%%%%%%%%%%%%%%%%%%%%

\section{$\mathcal F_1$-flat and $\mathcal F_1^{\fp}$-flat modules}

	This section introduces and studies the notions of $\mathcal F_1$-flat and $\mathcal F_1^{\fp}$-flat modules as being the $\Tor$-orthogonal classes $\mathcal F_1^{\top}$ and $\mathcal F_1^{\fp\top}$ of $\mathcal F_1$ and $\mathcal F_1^{\fp}$. Observe that $\mathcal F_1\mathcal F(R)\subseteq\mathcal F_1^{\fp}\mathcal F(R)$. We mainly seek conditions on rings $R$ for which the two notions of $\mathcal F_1$-flat module and $\mathcal F_1^{\fp}$-flat module collapse. In particular, we prove that if $\lim\limits_{\lr}\mathcal F_1^{\fp}=\mathcal F_1$, then $\mathcal F_1\mathcal F(R)=\mathcal F_1^{\fp}\mathcal F(R)$. Also, we exhibit an example $R$ such that $\lim\limits_{\lr}\mathcal F_1^{\fp}\neq\mathcal F_1$ showing that, in general, the concerned classes $\mathcal F_1\mathcal F(R)$ and $\mathcal F_1^{\fp}\mathcal F(R)$ are different.\\
	
	Let $\mathcal C$ be a class of right $R$-modules and $\mathcal D$ be a class of left $R$-modules. We put $$\mathcal C^\top=\ker\Tor_1^R(\mathcal C,\text -)=\{\text {left } R\text {-modules } M:\Tor_1^R(C,M)=0 \text{ for all }C\in\mathcal C\}$$ and $$^\top\mathcal D=\ker\Tor_1^R(\text -,\mathcal D)=\{\text {right }R\text{-modules } N:\Tor_1^R(N,D)=0 \text{ for all }D\in\mathcal D\}.$$
	 A pair $(\mathcal A,\mathcal B)$ of classes of $R$-modules is called a $\Tor$-torsion theory if $\mathcal A=$ $^\top\mathcal B$ and $\mathcal B=\mathcal A^\top$. Let $\mathcal C$ be a class of right $R$-modules. Then it is easy to check that $(^\top(\mathcal C^\top),\mathcal C^\top)$ is a Tor-torsion theory. Also, we put $\widehat{\mathcal C}:=$ $^\top(\mathcal C^\top)$. Note that $\lim\limits_{\lr}\mathcal C\subseteq \widehat{\mathcal C}$ as $\widehat {\mathcal C}$ is stable under direct limits. A $\Tor$-torsion theory $(\mathcal A,\mathcal B)$ is said to be generated by $\mathcal C$ if $A=\widehat{\mathcal C}$ (and thus $B=\mathcal C^\top$). Let $(\mathcal A_1,\mathcal B_1)$ and $(\mathcal A_2,\mathcal B_2)$ two $\Tor$-torsion theories generated by $\mathcal C_1$ and $\mathcal C_2$, respectively. Then the two pairs $(\mathcal A_1,\mathcal B_1)$ and $ (\mathcal A_2,\mathcal B_2)$ coincide if and only if $\widehat {\mathcal C_1}=\widehat {\mathcal C_2}$.\\
	 
	 We begin by proving the following lemma of general interest.
	 
	 \begin{lem}\label{2.1}
	 	Let $R$ be a ring and $\mathcal C$ and $\mathcal D$ be classes of right $R$-modules.\\
	 	1) $(\lim\limits_{\lr}\mathcal C)^{\top}=\widehat{\mathcal C}^{\top}=\mathcal C^{\top}$.\\
	 	2) If $\mathcal C\subseteq\mathcal D\subseteq \widehat {\mathcal C}$, then $\widehat{\mathcal C}=\widehat{\mathcal D}$.\\
	 	3) If $\lim\limits_{\lr}\mathcal C=\lim\limits_{\lr}\mathcal D$, then $\mathcal C^{\top}=\mathcal D^{\top}$ and  $\widehat{\mathcal C}=\widehat{\mathcal D}$.
	 \end{lem} 
 
 \begin{proof} 1) Note that $\widehat{\mathcal C}^{\top}=\mathcal C^{\top}$ and that $\widehat{\mathcal C}^{\top}\subseteq (\lim\limits_{\lr}\mathcal C)^{\top}\subseteq \mathcal C^{\top}$ as $\mathcal C\subseteq \lim\limits_{\lr}\mathcal C \subseteq \widehat {\mathcal C}$. Then the result easily follows. \\
2) Assume that $\mathcal C\subseteq\mathcal D\subseteq \widehat {\mathcal C}$. Then $\widehat {\mathcal C}\subseteq \widehat {\mathcal D}\subseteq \widehat {\widehat {\mathcal C}}$. Now, as $\widehat {\widehat {\mathcal C}}=\widehat {\mathcal C}$, we get $\widehat {\mathcal C}=\widehat {\mathcal D}$, as desired.\\
3)  It follows easily from (1).\end{proof}

\begin{defin}
	1) A left $R$-module $M$ is said to be $\mathcal F_1$-flat if
	$\Tor^1_{R}(H,M)=0$ for each right module $H\in \mathcal F_{1}$, that is, $M\in \mathcal F_1^\top$. The class of all left $\mathcal F_1$-flat modules is denoted by $\mathcal F_1\mathcal F(R)$.\\
	2) A left $R$-module $M$ is said to be $\mathcal F_1^{\fp}$-flat if
	$\Tor^1_{R}(H,M)=0$ for each right module $H\in \mathcal F_1^{\fp}$, that is, $M\in\mathcal F_1^{\fp\top}$. The class of all left $\mathcal F_1^{\fp}$-flat modules is denoted by $\mathcal F_1^{\fp}\mathcal F(R)$.
\end{defin}

%\begin{prop}
%	Let $R$ be a ring. Then\\
%	1) Any element $M\in\mathcal P_1^{\fp}$ is $2$-presented.\\ 
%	2) Any element $M\in \mathcal P_1$ is a direct summand of a direct limit $\lim\limits_{\lr} M_i$ such that $M_i\in$ FP-$ \mathcal P_1$ each $i\in I$.
%\end{prop}

%\begin{proof} 1) It is clear.\\
%	2) Let $H\in \mathcal P_1$. Then there exists a short exact sequence $0\longrightarrow P\stackrel j\longrightarrow L\stackrel f\longrightarrow M\longrightarrow 0$ such that $P$ is projective and $L$ is free. Since there exists  a free module $Q$ (a free complement) such that $P\oplus Q$ is free over $R$, we get the following exact sequence $$0\longrightarrow P\oplus Q=:L_0\longrightarrow L\oplus Q=:L_1\stackrel f\longrightarrow M\oplus Q=:N\longrightarrow 0$$ with $L_0,L_1$ are free over $R$. Note that $L_1$ is a direct union of its finitely generated free submodules $(L_{1i})_{i\in I}$. Let $(L_{0i})_i$ be finitely generated free $R$-modules such that $j(L_{0i})\subseteq L_{1i}$ for each $i$ in such way to get the following short exact sequence $$0\longrightarrow L_{0i}\stackrel j\longrightarrow L_{1i}\stackrel f\longrightarrow N_i:=f(L_{1i})\cong\dfrac {L_{1i}}{L_{0i}}\longrightarrow 0$$ for each $i\in I$. Note that each $N_i$ is an element of $\mathcal P_1$ which is finitely presented. Moreover, it is easy to check that $N=\lim\limits_{\longrightarrow} N_i$  completing the proof.
%\end{proof}

Next, we list some properties of $\mathcal F_1$-flat and $\mathcal F_1^{\fp}$-flat modules.

\begin{prop}\label{2.2} Let $R$ be a ring. Then\\
	1)  $\mathcal F_1\mathcal F(R)\subseteq \mathcal F_1^{\fp}\mathcal F(R)$.\\
		2) $\mathcal F_1\mathcal F(R)$ and $\mathcal F_1^{\fp}\mathcal F(R)$ are stable under direct sums and direct limits.\\
	3) $\mathcal F_1\mathcal F(R)$ and $\mathcal F_1^{\fp}\mathcal F(R)$ are stable under submodules.\\
	4) Any left ideal of $R$ is $\mathcal F_1$-flat and $\mathcal F_1^{\fp}$-flat.
\end{prop}

\begin{proof}
	1) and 2) are clear as the functor $\Tor_n^R (H,-)$ commutes with direct sums and direct limits for any right $R$-module $H$ and each positive integer $n$.\\ 
	3) Let $N$ be a submodule of a left $\mathcal F_1$-flat module $M$. Let $H\in\mathcal F_1$be a right module and consider the short exact sequence $0\longrightarrow N\longrightarrow M\longrightarrow \dfrac MN\longrightarrow 0$ of left modules. Then applying the functor $H\otimes_R-$, we get the exact sequence $$\Tor_2^R\Big (H,\dfrac MN\Big )\longrightarrow \Tor_1^R(H,N)\longrightarrow \Tor_1^R(H,M).$$ Now, as $\Tor_1^R(H,M)=0$ since $M$ is $\mathcal F_1$-flat and $\Tor_2^R\Big (H,\dfrac MN\Big )=0$ as fd$_R(H)\leq 1$, we deduce that $\Tor_1^R(H,N)=0$. Therefore $N$ is a $\mathcal F_1$-flat left $R$-module, as desired.\\
	4) It follows from 3).
\end{proof}

\begin{prop}\label{2.3}  Let $R$ be a ring and $M$ a left $R$-module. Then the following assertions are equivalent:
	
	1) $M$ is $\mathcal F_1$-flat (resp., $\mathcal F_1^{\fp}$-flat);
	
	2) Each finitely generated submodule of $M$ is $\mathcal F_1$-flat (resp., $\mathcal F_1^{\fp}$-flat);
	
	3) $M=\lim\limits_{\rightarrow}M_i$ such that each $M_i$ is a finitely generated $\mathcal F_1$-flat module (resp., $\mathcal F_1^{\fp}$-flat module).\\
	 
\end{prop}

\begin{proof}
1) $\Rightarrow$ 2) It is direct as $\mathcal F_1\mathcal F(R)$ is stable under submodules.\\
2) $\Rightarrow$ 3) It is easy as $M$ is direct limit of the direct set consisting of its submodules.\\
3) $\Rightarrow$ 1) It holds as $\mathcal F_1\mathcal F(R)$ is stable under direct limits.
\end{proof}

\begin{prop}\label{2.4} Let $R$ be a ring. Then\\
1)	The pair $(\mathcal F_1,\mathcal F_1\mathcal F(R))$ is a $\Tor$-torsion theory.\\
2) $(\widehat {\mathcal F_1^{\fp}},\mathcal F_1^{\fp}\mathcal F(R))$ is a $\Tor$-torsion theory with $\lim\limits_{\lr}\mathcal F_1^{\fp}\subseteq \widehat {\mathcal F_1^{\fp}}\subseteq \mathcal F_1$.
\end{prop}

\begin{proof}
1)	It suffices to prove that $^{\top}\mathcal F_1\mathcal F(R)=\mathcal F_1$. First, note that $\mathcal F_1\subseteq$ $^{\top}\mathcal F_1\mathcal F(R)$. Conversely, let $H\in$ $^{\top}\mathcal F_1\mathcal F(R)$. Then $\Tor_1^R(H,M)=0$ for any $M\in\mathcal F_1\mathcal F(R)$. Now, let $I$ be a left ideal of $R$. Then, as $I$ is $\mathcal F_1$-flat, we get $\Tor_1^R(H,I)=0$. Therefore $\Tor_2^R\Big (H,\dfrac RI\Big )=0$ for each left ideal $I$ of $R$. It follows that $\fd_R(H)\leq 1$ (see \cite[Lemma 9.18]{R}), that is, $H\in \mathcal F_1$, as desired.\\
2) It is straighforward in view of the above discussion.
\end{proof}

Similarly, next we introduce the class of $\mathcal P_1$-flat modules and the class of $\mathcal P_1^{\fp}$-flat modules.

\begin{defin} Let $R$ be a ring.\\
	1) A left $R$-module $M$ is said to be $\mathcal P_1$-flat if
	$\Tor^1_{R}(H,M)=0$ for each right module $H\in \mathcal P_{1}$, that is, $M\in \mathcal P_1^\top$. The class of all left $\mathcal P_1$-flat modules is denoted by $\mathcal P_1\mathcal F(R)$.\\
	2) A left $R$-module $M$ is said to be $\mathcal P_1^{\fp}$-flat if
	$\Tor^1_{R}(H,M)=0$ for each right module $H\in \mathcal P_1^{\fp}$, that is, $M\in\mathcal P_1^{\fp\top}$. The class of all $\mathcal P_1^{\fp}$-flat modules is denoted by $\mathcal P_1^{\fp}\mathcal F(R)$.
\end{defin}

It is known that any module is a direct limit of a direct system of finitely presented modules. One then may naturally wonder whether $\lim\limits_{\lr}\mathcal F_1^{\fp}=\mathcal F_1(=\lim\limits_{\lr}\mathcal F_1)$ in the general case. A positive answer to this question yields ultimately that the two notions of $\mathcal F_1$-flat module and $\mathcal F_1^{\fp}$-flat module collapse (see Corollary \ref{2.8}). In connection with this problem, it is worth recalling that in Problem 22 of \cite{FS}, page 246, Fuchs and Salce asked for the structure of $\lim\limits_{\lr}\mathcal P_1$, the set of modules which are direct limit of modules in $\mathcal P_1$. It is known that if $R$ is a P\"rufer domain, then $\lim\limits_{\lr}\mathcal P_1=\Mod(R)(=\mathcal F_1)$. Subsequently, Hugel and Trlifaj generalized this result by proving that if $R$ is an integral domain, then $\lim\limits_{\lr}\mathcal P_1=\lim\limits_{\lr}\mathcal P_1^{\fp}=\mathcal F_1$ \cite[Theorem 2.3]{HT}. It appears then legitimate to ask whether $\lim\limits_{\lr}\mathcal P_1=\mathcal F_1$ in the general context of an arbitrary ring $R$. In this regard, in \cite{BH}, Bazzoni and Herbera gave a negative answer to this question by providing an example of a commutative Noetherian ring $R$ (namely $R:=\dfrac {k[x,y,z]}{(z^2,zx,zy)}$ with $x,y,z$ are indeterminates over a field $k$) such that $\lim\limits_{\lr}\mathcal P_1\neq\mathcal F_1$ \cite[Example 8.5]{BH}. We prove ahead, see Corollary \ref{3.3.1}, that this very ring $R$ do permit to answer negatively the above question, namely, that $\lim\limits_{\lr}\mathcal F_1^{\fp}\neq\mathcal F_1$. This shows in particular that the two notions of $\mathcal F_1$-flat module and $\mathcal F_1^{\fp}$-flat module are distinct. 

This brief discussion shows our interest in seeking connections between the two equalities $\lim\limits_{\lr}\mathcal F_1^{\fp}=\mathcal F_1$ and $\lim\limits_{\lr}\mathcal P_1=\mathcal F_1$. Particularly, we prove that $\lim\limits_{\lr}\mathcal P_1=\mathcal F_1$ implies $\lim\limits_{\lr}\mathcal F_1^{\fp}=\mathcal F_1$ and that when $R$ is $\mathcal F_1^{\fp}$-coherent (see Definition \ref{3.0}), then $\lim\limits_{\lr}\mathcal P_1=\mathcal F_1$ if and only if $\lim\limits_{\lr}\mathcal F_1^{\fp}=\mathcal F_1$.

%Our aim via our next results is to generalize the above result on $\lim\limits_{\lr}\mathcal P_1$ in the context of Pr\"ufer domains to semi-hereditary rings and to seek $\lim\limits_{\lr}\mathcal P_1$ for semi-regular rings.

%\begin{prop}\label{1.4}
%	Let $R$ be a ring. If $\lim\limits_{\lr}\mathcal %F_1^{\fp}=\mathcal F_1$, then $\mathcal F_1^{\fp}\mathcal %F(R)=\mathcal F_1\mathcal F(R)$.
%\end{prop}

%\begin{proof}
%	Assume that $\lim\limits_{\lr}\mathcal F_1^{\fp}=\mathcal F_1$. Let $M\in\mathcal F_1^{\fp}\mathcal F(R)$ and let $N\in\mathcal F_1$. Then $N\in\lim\limits_{\lr}\mathcal F_1^{\fp}$. As $\Tor_1^R(-,M)$ commutes with direct limits and $M\in \mathcal F_1^{\fp}\mathcal F(R)$, it follows that $\Tor_1^R(N,M)=0$. Hence $M\in\mathcal F_1\mathcal F(R)$. Consequently, $\mathcal F_1^{\fp}\mathcal F(R)\subseteq\mathcal F_1\mathcal F(R)$. Apply Proposition \ref{1.1}(1) to complete the proof.
%\end{proof}

\begin{prop}\label{2.5}
	Let $R$ be a ring. Then\\
	1) $\lim\limits_{\lr}\mathcal P_1^{\fp}=\lim\limits_{\lr}\mathcal P_1$.\\
	2) $\lim\limits_{\lr}\mathcal P_1\subseteq \lim\limits_{\lr}\mathcal F_1^{\fp}\subseteq \mathcal F_1$.
%	3) If $\lim\limits_{\lr}\mathcal P_1=\mathcal F_1$, then $\lim\limits_{\lr}\mathcal F_1^{\fp}=\mathcal F_1$.
\end{prop}

\begin{proof} 1) Let $\mathcal P_1^\infty$ denote the class of all elements of $\mathcal P_1$ which admit projective resolutions consisting of finitely generated projective modules. It is well known that $\mathcal P_1\subseteq\lim\limits_{\lr}\mathcal P_1^\infty$ (see \cite[page 12]{BH}). Since  $\mathcal P_1^\infty\subseteq \mathcal P_1^{\fp}$, then $\lim\limits_{\lr}\mathcal P_1^\infty\subseteq \lim\limits_{\lr}\mathcal P_1^{\fp}$ and thus, we get $\mathcal P_1\subseteq \lim\limits_{\lr}\mathcal P_1^{\fp}$. By \cite[Lemma 1.2]{HT}, $\lim\limits_{\lr}\mathcal P_1^{\fp}$ is closed under direct limit. Hence $\lim\limits_{\lr}\mathcal P_1\subseteq \lim\limits_{\lr}\mathcal P_1^{\fp}$. Now, since $\lim\limits_{\lr}\mathcal P_1^{\fp}\subseteq \lim\limits_{\lr}\mathcal P_1$ (as $\mathcal P_1^{\fp}\subseteq\mathcal P_1$), it follows that $\lim\limits_{\lr}\mathcal P_1^{\fp}=\lim\limits_{\lr}\mathcal P_1$, as desired.  \\ 
	2) Consider the following inclusions $$\mathcal P_1^{\fp}\subseteq \mathcal F_1^{\fp}\subseteq \mathcal F_1.$$ Then $\lim\limits_{\lr}\mathcal P_1^{\fp}\subseteq \lim\limits_{\lr}\mathcal F_1^{\fp}\subseteq \lim\limits_{\lr}\mathcal F_1=\mathcal F_1$. By (1), $\lim\limits_{\lr}\mathcal P_1^{\fp}=\lim\limits_{\lr}\mathcal P_1$. It follows that $\lim\limits_{\lr}\mathcal P_1\subseteq\lim\limits_{\lr}\mathcal F_1^{\fp}\subseteq \mathcal F_1$, as contended.\\
%	3) It follows from (2).
\end{proof}

\begin{cor}\label{2.6}
	Let $R$ be a ring. If $\lim\limits_{\lr}\mathcal P_1=\mathcal F_1$, then $\lim\limits_{\lr}\mathcal F_1^{\fp}=\mathcal F_1$.
\end{cor}

\begin{prop}\label{2.7} Let $R$ be a ring. Then\\
	1)	The pair $(\widehat{\mathcal P_1},\mathcal P_1\mathcal F(R))$ is a $\Tor$-torsion theory.\\
	2) $(\widehat {\mathcal P_1^{\fp}},\mathcal P_1^{\fp}\mathcal F(R))$ is a $\Tor$-torsion theory with $$\lim\limits_{\lr}\mathcal P_1\subseteq\lim\limits_{\lr}\mathcal F_1^{\fp}\subseteq \widehat {\mathcal F_1^{\fp}}\subseteq \mathcal F_1\text { and }\lim\limits_{\lr}\mathcal P_1=\widehat {\mathcal P_1^{\fp}}=\widehat{\mathcal P_1}\subseteq \mathcal F_1.$$
	3) $(\widehat{\mathcal P_1},\mathcal P_1\mathcal F(R))=(\widehat {\mathcal P_1^{\fp}},\mathcal P_1^{\fp}\mathcal F(R))$.
\end{prop}

\begin{proof}
	1) It is direct.\\
	2) The inequalities $\lim\limits_{\lr}\mathcal P_1\subseteq\lim\limits_{\lr}\mathcal F_1^{\fp}\subseteq \widehat {\mathcal F_1^{\fp}}\subseteq \mathcal F_1$ holds by Proposition \ref{2.5} and note that $$\mathcal P_1^{\fp}\subseteq \mathcal P_1\subseteq \lim\limits_{\lr}\mathcal P_1=\lim\limits_{\lr}\mathcal P_1^{\fp}\subseteq \widehat {\mathcal P_1^{\fp}}\subseteq \widehat {\mathcal P_1}\subseteq \mathcal F_1.$$ Then, by Lemma \ref{2.1}, $\widehat {\mathcal P_1^{\fp}}=\widehat {\mathcal P_1}$. Moreover, by \cite[Theorem 2.3]{HT}, $\lim\limits_{\lr}\mathcal P_1^{\fp}=\widehat {\mathcal P_1^{\fp}}$. Then we are done.\\
	3) It is direct using (2).
\end{proof}

We next list immediate consequences of Proposition \ref{2.7}.

\begin{cor}\label{2.8}
	Let $R$ be a ring. If $\lim\limits_{\lr}\mathcal F_1^{\fp}=\mathcal F_1$, then $\mathcal F_1^{\fp}\mathcal F(R)=\mathcal F_1\mathcal F(R)$.\end{cor}

\begin{proof} Assume that $\lim\limits_{\lr}\mathcal F_1^{\fp}=\mathcal F_1$. Then, by Proposition \ref{2.7} (2), $\widehat {\mathcal F_1^{\fp}}=\mathcal F_1$. Hence $\mathcal F_1^{\fp}\mathcal F(R)=\mathcal F_1\mathcal F(R)$.

\end{proof}

\begin{cor} \label{2.9} Let $R$ be a ring. If $\lim\limits_{\lr}\mathcal P_1=\mathcal F_1$, then
 $\mathcal P_1\mathcal F(R)=\mathcal F_1^{\fp}\mathcal F(R)=\mathcal F_1\mathcal F(R).$
	\end{cor}

\begin{proof}
	Note that $\mathcal F_1\mathcal F(R)\subseteq\mathcal F_1^{\fp}\mathcal F(R)\subseteq \mathcal P_1^{\fp}\mathcal F(R)=\mathcal P_1\mathcal F(R)$ as $\mathcal P_1^{\fp}\subseteq\mathcal F_1^{\fp}\subseteq \mathcal F_1$. Assume that $\lim\limits_{\lr}\mathcal P_1=\mathcal F_1$. Then, by Lemma \ref{2.1}(3), we get $\mathcal F_1\mathcal F(R)=\mathcal P_1\mathcal F(R)$. Also, by Corollary \ref{2.6}, $\lim\limits_{\lr}\mathcal F_1^{\fp}=\mathcal F_1$. Thus, by Corollary \ref{2.8}, $\mathcal F_1^{\fp}\mathcal F(R)=\mathcal F_1\mathcal F(R)$. Then the desired result follows.
\end{proof}

\begin{cor}\label{2.10}
	Let $R$ be a ring. Then the following assertions are equivalent:
	
	1) $\lim\limits_{\lr}\mathcal P_1=\lim\limits_{\lr}\mathcal F_1^{\fp}$;
	
	2) $\mathcal P_1\mathcal F(R)=\mathcal F_1^{\fp}\mathcal F(R)$;
	
	3) $\widehat {\mathcal P_1}=\widehat {\mathcal F_1^{\fp}}$.
\end{cor}

\begin{proof} It is clear that (2) $\Leftrightarrow$ (3).\\
	1) $\Rightarrow$ 3) Apply Lemma \ref{2.1}(3).\\
	3) $\Rightarrow$ 1) By Proposition \ref{2.7}, $\widehat {\mathcal P_1}=\widehat {\mathcal P_1^{\fp}}$. Then $\widehat {\mathcal P_1^{\fp}}=\widehat {\mathcal F_1^{\fp}}$. Also, By Proposition \ref{2.7}, $\lim\limits_{\lr}\mathcal P_1^{\fp}=\widehat {\mathcal P_1^{\fp}}$ and $\lim\limits_{\lr}\mathcal P_1^{\fp}\subseteq\lim\limits_{\lr}\mathcal F_1^{\fp}\subseteq \widehat {\mathcal F_1^{\fp}}$. It follows that $\lim\limits_{\lr}\mathcal P_1=\lim\limits_{\lr}\mathcal F_1^{\fp}$, as desired. 
	
	\end{proof}

The following results provide a series of rings which satisfy the equality $\lim\limits_{\lr}\mathcal P_1=\mathcal F_1$.

\begin{prop}\label{2.11} Let $R$ be a ring. If $R$ is a right semi-hereditary ring, then  $$\lim\limits_{\lr}\mathcal P_1=\mathcal F_1=\Mod(R).$$
\end{prop}	

\begin{proof}
	Assume that $R$ is a right semi-hereditary. Let $M\in\Mod(R)$. Note that $M$ is a direct limit of finitely presented modules $(M_i)_i$. Also, note that, as $R$ is semi-hereditary, any finitely presented module is an element of $\mathcal P_1$ since any finitely generated submodule of a projective module is projective. Hence each $M_i\in \mathcal P_1$ and thus $M\in\lim\limits_{\lr}\mathcal P_1$. It follows that $\lim\limits_{\lr}\mathcal P_1=\Mod(R)=\mathcal F_1$. This completes the proof.
\end{proof}

Let $R$ be a ring. Recall that the right finitistic projective dimension of $R$ is the positive integer $\FPD(R)=\max\{\pd_R(M):M\in\Mod(R)$ and $\pd_R(M)<+\infty\}$ and the right finitistic flat dimension of $R$ is the positive integer $\FFD(R)=\max\{\fd_R(M):M\in\Mod(R)$ and $\fd_R(M)<+\infty\}$.

%Observe that if $M$ is an $R$-module such that $\pd_R(M)<+\infty$, then $\fd_R(M)<+\infty$. This permits to introduce the  following weaker finitistic dimension of $R$ as follows: the projective finitistic flat dimension of a ring $R$ is the following invariant of $R$  $$\pFFD:=\max\{\fd_R(M):M\in\Mod(R)\text{ and }\pd_R(M)<+\infty\}.$$ Note that $\pFFD(R)\leq\FFD(R)$ and $\pFFD(R)\leq\FPD(R)$. The next result characterizes the rings for which $\lim\limits_{\lr}\mathcal P_1=\mathcal F(R)$ by the vanishing of the projective finitistic flat dimension.

%\begin{prop}
%	Let $R$ be a ring. Then the following assertions are equivalent.
	
%	1) $\lim\limits_{\lr}\mathcal P_1=\mathcal F(R)$;
	
%	2) $\mathcal P_1\subseteq \mathcal F(R)$;
	
%	3) $\{M\in \Mod(R):\pd_R(M)<+\infty\}\subseteq \mathcal F(R)$;
	
%	4) $\pFFD(R)=0$. 
%\end{prop}

\begin{prop}\label{2.12}
	Let $R$ be a ring such that $\FFD(R)=0$. Then $$\lim\limits_{\lr}\mathcal P_1=\mathcal F_1=\mathcal F(R).$$
\end{prop}

\begin{proof} First, note that $\mathcal F(R)\subseteq\lim\limits_{\lr}\mathcal P_1\subseteq \mathcal F_1$ as any flat module is a direct limit of a direct system of finitely generated free modules. As $\FFD(R)=0$, $\mathcal F(R)=\mathcal F_1$ and thus $\lim\limits_{\lr}\mathcal P_1=\mathcal F_1=\mathcal F(R)$, as desired.
	
\end{proof}

The next result sharpens Proposition \ref{2.12} in the case where $R$ is a commutative ring.

\begin{prop}\label {2.12.1} Let $R$ be a commutative ring and $Q$ its classical ring of quotients.\\
1)	If $\FFD(Q)=0$, then $\lim\limits_{\lr}\mathcal P_1=\mathcal F_1$.\\
2) Moreover, if $Q$ is Noetherian, then the following assertions are equivalent: 

i) $\lim\limits_{\lr}\mathcal P_1=\mathcal F_1$;

ii) $\FFD(Q)=0$;

iii) $\depth(Q_p)=0$ for each prime ideal $p$ of $Q$.
\end{prop}

\begin{proof} 1) It holds by \cite[Corollary 6.8]{BH}.\\
	2) First, note that \cite[Theorem 2.4]{AB} ensures the equivalence ii) $\Leftrightarrow iii)$. Applying \cite[ Corollary 6.8 and Lemma 8.3]{BH}, we get the remaining equivalence i) $\Leftrightarrow$ ii), as desired.
\end{proof}

Recall that, in 1982, Matlis proved that a ring $R$ is coherent if and only if $\Hom_R(M,N)$ is flat for any injective $R$-modules $M$ and $N$ \cite{M1}. Also, in 1985, he introduced the notion of semi-coherent commutative ring. In effect,
he defined a ring $R$ to be semi-coherent if it is commutative and $\Hom_R(M,N)$ is a submodule of a flat $R$-module for any injective $R$-modules $M$ and $N$. Then, inspired by this definition and by von Neumann regularity, he defined a ring to be semi-regular if it is commutative and if any module can be embedded in a flat module. He then provided a connection between this notion and coherence; namely, a commutative ring $R$ is semi-regular if and only if $R$ is coherent and $R_M$ is semi-regular for every maximal ideal $M$ of $R$. He also proved that a ring $R$ is a Pr\"ufer domain if and only if $\dfrac RI$ is a semi-regular ring for each nonzero finitely generated ideal $I$ of $R$. A semi-regular ring is also termed an IF-ring (a ring in which any injective module is flat). The class of semi-regular rings then includes von Neumann regular rings, Quasi-Frobenius rings and quotients of Pr\"ufer domains by nonzero finitely generated ideals. 

\begin{cor}\label{2.13}
	Let $R$ be a semi-regular ring. Then $\FFD(R)=0$ and thus $$\lim\limits_{\lr}\mathcal P_1=\mathcal F_1=\mathcal F(R).$$
\end{cor}

\begin{proof}
	Let $M\in \mathcal F_1$ and let $K$ be an $R$-module. As $R$ is semi-regular, there exists an exact sequence $0\lr K\lr F\lr N=\dfrac FK\lr 0$ such that $F$ is a flat module. Then $\Tor_1^R(M,K)\cong\Tor_2^R(M,N)=0$ as $M\in\mathcal F_1$. Therefore $M\in\mathcal F(R)$. Hence $\mathcal F_1=\mathcal F(R)$ and thus $\FFD(R)=0$. Now, Proposition \ref{2.12} completes the proof.
\end{proof}

\begin{cor}\label{2.14}
	Let $R$ be a Pr\"ufer domain and $I$ a finitely generated ideal of $R$. Then $$\lim\limits_{\lr}\mathcal P_1\Big (\dfrac RI\Big )=\mathcal F_1\Big (\dfrac RI\Big ).$$
\end{cor}

\begin{proof} If $I=0$, then  $\lim\limits_{\lr}\mathcal P_1=\Mod(R)=\mathcal F_1$. If $I\neq 0$, then $\dfrac RI$ is semi-regular and thus Corollary \ref{2.13} yields the desired result. 
\end{proof}

\begin{cor}\label{2.15} Let $R$ be a ring. Then\\
1) If $R$ is perfect, then $\lim\limits_{\lr}\mathcal P_1=\mathcal F_1$.\\
2)  If $R$ is an integral domain, then $\lim\limits_{\lr}\mathcal P_1=\mathcal F_1$.
\end{cor}

\begin{proof}
	1) Assume that $R$ is perfect. Then $\mathcal F_1=\mathcal P_1$ and thus $\lim\limits_{\lr}\mathcal P_1=\lim\limits_{\lr}\mathcal F_1=\mathcal F_1.$ \\
	2) It is direct applying Proposition \ref{2.12.1} as the classical ring of quotients of an integral domain coincides with its quotient field or use \cite[Theorem 3.5]{HT}.
\end{proof}

\begin{cor}\label{2.16}
	Let $R$ be a ring. If $R$ is either semi-hereditary, perfect or an integral domain, then $$\mathcal P_1\mathcal F(R)=\mathcal F_1^{\fp}\mathcal F(R)=\mathcal F_1\mathcal F(R).$$
\end{cor}

\begin{proof}
	Combine Corollary \ref{2.9}, Proposition \ref{2.11} and Corollary \ref{2.15}.
\end{proof}

\begin{prop}\label{2.17} Let $R$ be an integral domain. Then the following assertions are equivalent:
	
1) $M$ is $\mathcal F_1$-flat;

2) $M$ is $\mathcal F_1^{\fp}$-flat;

3) $M$ is torsion-free.
	\end{prop}

\begin{proof} It follows from Corollary \ref{2.16} and \cite[Lemma 2.3]{L}.
	\end{proof}

The remaining results of this section deal with characterizing the rings of weak global dimension less than or equal to 1 in terms of the behavior of the $\mathcal F_1$-flat modules.

\begin{prop}\label{2.18} Let $R$ be a ring such that $\wdim(R)\leq 1$. Then  $$\lim\limits_{\lr}\mathcal F_1^{\fp}=\mathcal F_1.$$ If moreover $R$ is right coherent, then $\lim\limits_{\lr}\mathcal P_1=\mathcal F_1.$
\end{prop}	

\begin{proof}
	Since $\wdim(R)\leq 1$, then $\mathcal F_1=\Mod(R)$ and $\mathcal F_1^{\fp}$ is the class of all finitely presented modules. Now, as any module is direct limit of finitely presented modules, the result easily follows. Assume that $R$ is moreover right coherent. Then $\mathcal F_1^{\fp}=\mathcal P_1^{\fp}$. Hence $\lim\limits_{\lr}\mathcal P_1^{\fp}=\lim\limits_{\lr}\mathcal F_1^{\fp}=\mathcal F_1.$ It follows, by Proposition \ref{2.5}, that $\lim\limits_{\lr}\mathcal P_1=\mathcal F_1,$ as desired.
\end{proof}

\begin{cor}\label{2.19}
	Let $R$ be a ring such that $\wdim(R)\leq 1$. Then $$\mathcal F_1^{\fp}\mathcal F(R)=\mathcal F_1\mathcal F(R).$$
\end{cor}

\begin{proof}
	Combine Proposition \ref{2.18} and Corollary \ref{2.8}.
\end{proof}

\begin{prop} Let $R$ be a ring. Then the following assertions are equivalent.
	
	1) Any $\mathcal F_1^{\fp}$-flat module is flat;
	
	2) Any $\mathcal F_1$-flat module is flat;
	
	3) $\wdim(R)\leq 1$.
	
	\end{prop}

\begin{proof} 1) $\Rightarrow$ 2) is clear.\\
	2) $\Rightarrow$ 3) Assume that any $\mathcal F_1$-flat module is flat. Then $\mathcal F_1=\Mod(R)$ so that $\wdim(R)\leq 1$.\\
	3) $\Rightarrow$ 1) Assume that $\wdim(R)\leq 1$. Then, $\mathcal F_1=\Mod(R)$ and thus any $\mathcal F_1$-flat module is flat. Also,
by corollary \ref{2.19}, $\mathcal F_1\mathcal F(R)=\mathcal F_1^{\fp}\mathcal F(R)$. Therefore any $\mathcal F_1^{\fp}$-flat module is flat, as desired.
	
\end{proof}

\section{ Coherence with respect to $\mathcal F_1^{\fp}$-flat modules }

This section introduces and studies a new  class of rings that we term $\mathcal F_1^{\fp}$-coherent rings. The new concept of a $\mathcal F_1^{\fp}$-coherent ring turns out to possess exactly the same behavior vis-\`a-vis the $\mathcal F_1^{\fp}$-flat modules than the classical one of a coherent ring vis-\`a-vis the flat modules. \\

We begin by giving the definition of $\mathcal F_1^{\fp}$-coherent rings.

\begin{defin}\label{3.0} A ring
	$R$ is called right $\mathcal F_1^{\fp}$-coherent if any finitely generated flat right submodule of a free module is finitely presented (an thus projective).
\end{defin}

%\begin{prop}
%	Let $R$ be a ring. Then $R$ is a $\mathcal F$-coherent ring if and only if any finitely generated flat submodule of a finitely presented module is finitely presented (and thus projective).
%\end{prop}

%\begin{proof} Assume that $R$ is $\mathcal F$-coherent and let $M$ be a finitely generated flat submodule of a finitely presented $H$. 
	%Assume that any finitely generated flat submodule of a finitely presented module is projective.
%\end{proof}

The following proposition provides a bunch of examples of $\mathcal F_1^{\fp}$-coherent rings and records the fact that the $\mathcal F_1^{\fp}$-coherence permits to unify interesting classes of rings, namely semi-hereditary rings, coherent rings (and thus Noetherian rings), perfect rings and integral domains.

Recall that a commutative ring $R$ is said to be a semi-coherent ring if $\Hom_R(B,C)$ is a submodule of a flat $R$-module for any
pair of injective R-modules $B$ and $C$ \cite[Definition, page 344]{M1}. Commutative coherent rings and integral domains are semi-coherent rings \cite[Examples, page 345]{M1}.

\begin{prop}\label{3.1}
	1) Any right semi-hereditary ring is $\mathcal F_1^{\fp}$-coherent.\\
	2) Any right coherent ring, and thus any right Noetherian ring, is $\mathcal F_1^{\fp}$-coherent.\\
	3) Any right perfect ring is $\mathcal F_1^{\fp}$-coherent.\\
	4) Any semi-coherent ring is $\mathcal F_1^{\fp}$-coherent.\\
	5) Any integral domain is $\mathcal F_1^{\fp}$-coherent.
\end{prop}

\begin{proof}
	1), 2) and 3) are clear.\\
	4) It follows from \cite[Proposition 1.3]{M1}.\\
	5) It is direct as any integral domain is semi-coherent.
\end{proof}

%We continue with the following characterizations of rings such that $\mathcal F_1^{\fp}$$=\mathcal P_1^{\fp}$ in
%terms of, among others, $\mathcal F_1^{\fp}$-flat and $\mathcal F_1^{\fp}$-injective modules. Some characterizations of coherent rings given by Chase \cite{C} and Stenstom \cite{S} generalize easily to
%our case. We include the proof for completeness.

Next, we present the main theorem of this section. It sheds light on inherent properties of $\mathcal F_1^{\fp}$-coherent rings which are very similar to those of coherent rings. More precisely, it highlights the fact that $\mathcal F_1^{\fp}$-coherent rings possess exactly the same behavior with respect to $\mathcal F_1^{\fp}$-flat modules than coherent rings with respect to flat modules. In fact, recall that a ring $R$ is right coherent if and only if any direct product of flat left $R$-modules is flat if and only if any direct product of copies of $R$ is flat \cite[Theorem 2.1]{C}. Also, it is proved in \cite{CS} that $R$ is right coherent if and only if (a left $R$-module $M$ is flat $\Leftrightarrow$ $M^{++}$ is flat) \cite[Theorem 1]{CS}.\\

First, we establish the following lemmas.

\begin{lem}\label{3.1.1}
	Let $R$ be a ring. Then any product of $\mathcal P_1$-flat modules is $\mathcal P_1$-flat.
\end{lem}

\begin{proof} let $(M_{\alpha})_{\alpha\in\Lambda}$ be a family of $\mathcal P_1$-flat modules. Let $N\in$ $\mathcal P_1^{\fp}$. Then there exists an exact sequence $0 \longrightarrow K\longrightarrow L\longrightarrow N\longrightarrow 0$ such that $L$ is a finitely generated free $R$-module and $K$ is a finitely generated projective module, and thus $N$ is $2$-presented. Therefore, by \cite[Lemma 2.10]{CD}, $$\Tor_{1}^R(N,\prod M_{\alpha}) \cong \prod \Tor_{1}^R(N,M_{\alpha}).$$ Hence $\Tor_{1}^R(N,\prod M_{\alpha})=0$. It follows that $\prod M_{\alpha}$ is $\mathcal P_1^{\fp}$-flat, and thus, by Proposition \ref{2.7}(3), $\prod M_{\alpha}$ is $\mathcal P_1$-flat, as desired.  \\
	\end{proof}

\begin{lem} \label{3.1.2} Let $R$ and $S$ be rings and consider the situation $(A_R,$ $_SB_R,$ $_SC)$. If $A$ is a finitely presented $R$-module, then there is a natural isomorphism $$ A\otimes_R\Hom_S(B,C)\cong\Hom_S(\Hom_R(A,B),C).$$ 
	\end{lem}

\begin{proof}
	The proof is dual and similar to the proof of \cite[Lemma 3.59 and Lemma 3.60]{R}.
\end{proof}

The folowing result is dual to \cite[Lemma 2.7]{CD}.

\begin{lem} \label{3.1.3} Let $R$ and $S$ be rings and $n\geq 1$ a positive integer. Consider
	the situation $(A_R,$ $_SB_R,$ $_SC)$ with $A$ $n$-presented and $_SC$ injective.
	Then there is an isomorphism $$\Tor_{n-1}^R(A,\Hom_S(B,C))\cong\Hom_S(\Ext_R^{n-1}(A,B),C).$$\end{lem}

\begin{proof}
	The proof is based on Lemma \ref{3.1.2} and it is dual to the proof of \cite[Theorem 9.51 and the remark following it]{R}. 
\end{proof}

Next, we announce the main theorem of this section. Also, it is worth reminding the reader of the
useful adjointness isomorphism for derived functors
$$\mbox {Tor}_n^R(A,B)^+\cong\mbox {Ext}_R^n(A,B^+)$$ for any left $R$-module $B$,
any right $R$-module $A$ and any integer $n\geq 0$.

\begin{thm}\label{3.2}
	Let $R$ be a ring. Then the following are equivalent:
	
	1) $R$ is right $\mathcal F_1^{\fp}$-coherent;
	
	2) Any direct product  of $\mathcal F_1^{\fp}$-flat $R$-modules is $\mathcal F_1^{\fp}$-flat;
	
	3) Any direct product of flat modules is $\mathcal F_1^{\fp}$-flat;
	
	4) Any direct product of copies of $R$ is $\mathcal F_1^{\fp}$-flat;
	
%	4) The subcategory $\mathcal F_1^{\fp}\mathcal F(R)$ of $\Mod(R)$ is definable; 
	
	5) Any inverse limit of $\mathcal F_1^{\fp}$-flat modules is $\mathcal F_1^{\fp}$-flat;
	
	6) Any module $M\in \mathcal F_1^{\fp}$ is $2$-presented;
	
	7)  $\mathcal F_1^{\fp}=\mathcal P_1^{\fp}$;
	
	8) A left $R$-module $M$ is  $\mathcal F_1^{\fp}$-flat if and only if $M^{++}$ is $\mathcal F_1^{\fp}$-flat;
	
	9) Any torsionless left $R$-module is $\mathcal F_1^{\fp}$-flat.
	
	10) $\mathcal F_1^{\fp}\mathcal F(R)=\mathcal P_1\mathcal F(R)$;
	
	11) $\lim\limits_{\lr}\mathcal F_1^{\fp}=\lim\limits_{\lr}\mathcal P_1$;
	
	12) $\widehat {\mathcal P_1}=\widehat {\mathcal F_1^{\fp}}$.
	
%	Moreover, if $\lim\limits_{\lr}\mathcal P_1=\mathcal F_1$, then the above statements are equivalent to the following equivalent ones:
	
%	9)   Any direct product of $\mathcal F_1$-flat right modules is $\mathcal F_1$-flat;
	
%	10) Any direct product of flat right modules is $\mathcal F_1$-flat;
	
%	11) Any direct product of copies of $R$ is $\mathcal F_1$-flat;
	
%	12) Any inverse limit of $\mathcal F_1$-flat modules is $\mathcal F_1$-flat.	
\end{thm}

\begin{proof} 1) $\Rightarrow$ 2) Let $N\in$ $\mathcal F_1^{\fp}$. Then there exists an exact sequence $0 \longrightarrow K\longrightarrow L\longrightarrow N\longrightarrow 0$ such that $L$ is a finitely generated free $R$-module and $K$ is a finitely generated module, and thus $N$ is $2$-presented. Now, let $(M_{\alpha})_{\alpha\in\Lambda}$ be a family of $\mathcal F_1^{\fp}$-flat modules. Therefore, as $N$ is $2$-presented, by \cite[Lemma 2.10]{CD}, $$\Tor_{1}^R(N,\prod M_{\alpha}) \cong \prod \Tor_{1}^R(N,M_{\alpha}).$$ Hence (2) easily holds. \\
	2) $\Rightarrow$ 3) and 3) $\Rightarrow$ 4) are clear.\\
	2) $\Rightarrow$ 5) It is direct as an inverse limit of $\mathcal F_1^{\fp}$-flat modules is isomorphic to a submodule of the direct product of these modules and submodules of $\mathcal F_1^{\fp}$-flat modules are $\mathcal F_1^{\fp}$-flat.\\
	5) $\Rightarrow$ 2) It is straightforward as direct products are particular cases of inverse limits.\\  
	4) $\Rightarrow$ 6) Let $M\in$ $\mathcal F_1^{\fp}$ and let $I$ be an arbitrary set. Then there exists an exact sequence $0\longrightarrow K\longrightarrow R^n\longrightarrow M\longrightarrow 0$, where $K$ is a finitely generated flat $R$-module and $n\geq 1$ is an integer. Since $\prod_{i\in I}R$ is $\mathcal F_1^{\fp}$-flat and $M\in$ $\mathcal F_1^{\fp}$, we get $\Tor_1^R(M,\prod_{i\in I}R)=0$ and thus the sequence $$0\longrightarrow K\otimes_R\prod_{i\in I}R\longrightarrow R^n\otimes_R\prod_{i\in I}R\longrightarrow M\otimes_R\prod_{i\in I}R\longrightarrow 0$$ is exact. This induces the following commutative diagram:
	$$\xymatrix{0\ar[r]&K\otimes_{R}\prod_{i\in I}R\ar[r]\ar[d]^{\phi_{1}} &{R^n}\otimes_{R}\prod_{i\in I}R\ar[r]\ar[d]^{\phi_{2}} &M\otimes_{R}\prod_{i\in I}R\ar[r]\ar[d]^{\phi_{3}} &0\\
		0\ar[r]&K^I\ar[r]&{(R^n)}^I\ar[r]&M^I\ar[r]&0}$$
	since $\phi_{2}$ and $\phi_{3}$ are isomorphisms by \cite[Theorem 3.2.22]{EJ}, then  $\phi_{1}$ is isomorphism by the five lemma. Applying again \cite[Theorem 3.2.22]{EJ} yields that $K$ is finitely presented, so that $M$ is $2$-presented.\\ 
	6) $\Rightarrow$ 7) Clearly, $\mathcal P_1^{\fp}\subseteq \mathcal F_1^{\fp}$. Let $M\in\mathcal F_1^{\fp}$ and let $0\lr K\lr L\lr M\lr 0$ be an exact sequence such that $L$ is a finitely generated free module and $K$ is finitely presented. Now, since $M\in \mathcal F_1$, $K$ is a finitely presented flat module and thus a finitely generated projective module. It follows that $M\in\mathcal P_1^{\fp}$. Consequently, $\mathcal F_1^{\fp}=\mathcal P_1^{\fp}$.\\
	7) $\Rightarrow$ 1) Let $M$ be a finitely generated flat right submodule of a finitely generated free module $L$. Then, considering the exact sequence $0\lr M\lr L\lr K=\dfrac LM\lr 0$, it is clear that $K\in F_1^{\fp}$, so that $K\in\mathcal P_1^{\fp}$. It follows that $M$ is projective. Consequently, $R$ is right $\mathcal F_1^{\fp}$-coherent proving (1).\\
	6) $\Rightarrow$ 8)  As, by Proposition \ref{2.2}, $\mathcal F_1^{\fp}\mathcal F(R)$ is stable under taking submodules, it suffices to prove that if $M$ is $\mathcal F_1^{\fp}$-flat, then so is $M^{++}$ since $M$ is isomorphic to a submodule of $M^{++}$. Assume that $M$ is $\mathcal F_1^{\fp}$-flat. Let $H\in$ $\mathcal F_1^{\fp}$. Then, by (6), $H$ is $2$-presented. Hence, by Lemma \ref{3.1.3} and the above adjointness isomorphism for derived functors, $$\begin{array}{lll}\Tor_1^R(H,M^{++})&\cong&\Ext^1_R(H,M^+)^+\\
	&\cong&\Tor_1^R(H,M)^{++} \\
	&=&0.\end{array}$$ Therefore $M^{++}$ is $\mathcal F_1^{\fp}$-flat, as desired.\\
	8) $\Rightarrow$ 2) Asssume that (8) holds. Let $(M_{i})_{i}$ be a family of $\mathcal F_1^{\fp}$-flat $R$-modules. By  Proposition \ref{2.2}, $\oplus M_{i}$ is $\mathcal F_1^{\fp}$-flat, so 
	that, using (8),  $(\oplus M_{i})^{++}\cong (\prod\limits_i M_{i}^+)^+$ is $\mathcal F_1^{\fp}$-flat. Further, note that, by \cite[Lemma 1 (1)]{CS}, $\oplus M_{i}^+$ is a pure submodule of $\prod\limits_i M_{i}^+$ and thus the sequence $0\lr \oplus M_{i}^+\lr \prod\limits_i M_{i}^+\lr K\lr 0$, with $K:=\dfrac {\prod\limits_i M_{i}^+}{\oplus M_{i}^+}$, is pure exact. Hence the exact sequence of character modules $$0\lr K^+\lr (\prod\limits_i M_{i}^+)^+\longrightarrow (\oplus M_{i}^+)^+\longrightarrow 0$$ splits. Therefore, $(\oplus M_{i}^+)^+$, being isomorphic to a submodule of $(\prod\limits_i M_{i}^+)^+$, is $\mathcal F_1^{\fp}$-flat, by Proposition \ref{2.2}. Thus 
	$\prod\limits_i M_{i}^{++}\cong (\oplus M_{i}^+)^+$ is $\mathcal F_1^{\fp}$-flat. Since $\prod\limits_i M_{i}$ is a submodule of $\prod\limits_i M_{i}^{++}$, we get, by Proposition \ref{2.2}, that $\prod\limits_i M_{i}$ is $\mathcal F_1^{\fp}$-flat, as contended.\\
	
%	7) $\Rightarrow$ 1) Let $M$ be a finitely generated flat module such that $M\subseteq L$ with $L$ is a free module. First, note that $L$ is $\mathcal F_1^{\fp}$-flat and thus, by our hypotheses, $L^{++}$ is $\mathcal F_1^{\fp}$-flat. Now, as $M^{++}$ is isomorphic to a submodule of $L^{++}$, we get $M^{++}$ is $\mathcal F_1^{\fp}$-flat, by Proposition. Hence $M$ is $\mathcal F_1^{\fp}$-flat. Consider the exact sequence $0\longrightarrow M\longrightarrow L\longrightarrow K:=\dfrac LM\longrightarrow 0$. Then, as $M$ is flat, $K\in\mathcal F_1$. Hence $\Ext_R^1(K,M)=0$ as $M$ is $\mathcal F_1^{\fp}$-flat, so that the above exact sequence splits. This shows that $M$ is projective. It follows that $R$ is $\mathcal F_1^{\fp}$-coherent.\\
\noindent 4) $\Rightarrow$ 9) Let $M$ be a torsionless left module. Then there exists a set $I$ such that $M\subseteq R^I$. As $R^I$ is $\mathcal F_1^{\fp}$-flat, then $M$ is $\mathcal F_1^{\fp}$-flat.\\
9) $\Rightarrow$ 4) It is straightforward.\\
10)	$\Leftrightarrow$ 11) $\Leftrightarrow$ 12) hold by Corollary \ref{2.10}.\\
7) $\Rightarrow$ 11) It is direct as $\lim\limits_{\lr}\mathcal P_1^{\fp}=\lim\limits_{\lr}\mathcal P_1$.\\
10) $\Rightarrow$ 2) It suffices to apply Lemma \ref{3.1.1} completing the proof of the theorem.
	\end{proof} 

The next results discuss different properties and characteristics of right $\mathcal F_1^{\fp}$-coherent rings.

\begin{prop}\label{3.3}
	Let $R$ be a right $\mathcal F_1^{\fp}$-coherent ring. Then the following assertions are equivalent.
	
	1) $\mathcal F_1^{\fp}\mathcal F(R)=\mathcal F_1\mathcal F(R)$;
	
	2) $\mathcal P_1\mathcal F(R)=\mathcal F_1\mathcal F(R)$;
	
	3) $\widehat {\mathcal P_1}=\mathcal F_1$;
	
	4) $\lim\limits_{\lr}\mathcal F_1^{\fp}=\mathcal F_1$;
	
	5) $\lim\limits_{\lr}\mathcal P_1=\mathcal F_1$.
	
\end{prop}

\begin{proof} It is direct as, when $R$ is right $\mathcal F_1^{\fp}$-coherent, $\lim\limits_{\lr}\mathcal F_1^{\fp}=\lim\limits_{\lr}\mathcal P_1=\widehat {\mathcal P_1}$ and $\mathcal F_1^{\fp}\mathcal F(R)=\mathcal P_1\mathcal F(R)$. Besides, note that $\mathcal P_1\mathcal F(R)=\mathcal F_1\mathcal F(R)$ if and only if $\widehat {\mathcal P_1}=\mathcal F_1$.

\end{proof}

The following corollary shows that $\lim\limits_{\lr}\mathcal F_1^{\fp}\neq\mathcal F_1$, in general. The example given by Bazzoni and Herbera, namely, $R:=\dfrac {k[x,y,z]}{(z^2,zx,zy)}$ with $k$ is a field and $x,y,z$ are indeterminates over $k$ satisfying $\lim\limits_{\lr}\mathcal P_1\neq\mathcal F_1$ \cite[Example 8.5]{BH} is an evidence underpinning this fact. 

\begin{cor}\label{3.3.1}
	Let $R$ be a right coherent ring. Then the following assertions are equivalent:
	
	1) $\lim\limits_{\lr}\mathcal F_1^{\fp}=\mathcal F_1$;
	
	2) $\lim\limits_{\lr}\mathcal P_1=\mathcal F_1$.
\end{cor}

\begin{proof}
	It is straightforward as, by Proposition \ref{3.1}, $R$ is $\mathcal F_1^{\fp}$-coherent.
\end{proof}

The following result records that $\mathcal F_1^{\fp}$-coherent rings include all rings $R$ such that $\lim\limits_{\lr}\mathcal P_1=\mathcal F_1$.

\begin{cor}
	Let $R$ be a ring. If $\lim\limits_{\lr}\mathcal P_1=\mathcal F_1$, then $R$ is a right $\mathcal F_1^{\fp}$-coherent ring.
\end{cor}

\begin{proof} It follows easily from Theorem \ref{3.2} and Corollary \ref{2.6}.
\end{proof}

It is known that any semi-hereditary ring is of weak global dimension less than or equal to $1$. The converse do not hold in general. Next, we show that along with the $\mathcal F_1^{\fp}$-coherence the last statement ensures the semi-hereditarity of the ring.
	
\begin{prop}\label{3.4}
	Let $R$ be a ring. Then the following assertions are equivalent.
	
	1) $R$ is right semi-hereditary;
	
	2) $\wdim(R)\leq 1$ and $R$ is right $\mathcal F_1^{\fp}$-coherent.
\end{prop}

\begin{proof} 1) $\Rightarrow$ 2) It is clear as, by Proposition \ref{3.1}, any right semi-hereditary ring is right $\mathcal F_1^{\fp}$-coherent.\\
	2) $\Rightarrow$ 1) Let $I$ be a finitely generated ideal of $R$. As $\wdim(R)\leq 1$, then $I$ is flat. Now, since $R$ is right $\mathcal F_1^{\fp}$-coherent, we get that $I$ is projective, as desired.
\end{proof}

\begin{cor}\label{3.5}
	Let $R$ be a right $\mathcal F_1^{\fp}$-coherent ring. Then the following assertions are equivalent.
	
	1) $R$ is right semi-hereditary;
	
	1) $\wdim(R)\leq 1$.
\end{cor}

Corollary \ref{3.5} provides a criterium to build examples of rings which are not $\mathcal F_1^{\fp}$-coherent. In fact, it suffices to choose any ring $R$ such that $\wdim(R)\leq 1$ and $R$ is not right semi-hereditary.\\

The last theorems of this section characterize rings $R$ such that any left  $R$-module is $\mathcal F_1^{\fp}$-flat (resp., $\mathcal F_1$-flat). Note that, for an $R$-module $M$, $\E(M)$ denotes the injective envelope of $M$.

\begin{thm}\label{3.6}
	Let $R$ be a ring. Then the following assertions are equivalent.
	
	1) Any finitely generated $R$-module is $\mathcal F_1^{\fp}$-flat;
	
		2) Any $R$-module is $\mathcal F_1^{\fp}$-flat;
		
		3) Any injective module is $\mathcal F_1^{\fp}$-flat;
		
		4) Any quotient of $\mathcal F^{\fp}_1$-flat module is $\mathcal F^{\fp}_1$-flat;
		
		5) $R$ is right $\mathcal F_1^{\fp}$-coherent and $\E(S)$ is $\mathcal F_1^{\fp}$-flat for each simple module $S$;
		
	%	5)  $R$ is $\mathcal F_1^{\fp}$-coherent and any cocyclic module is $\mathcal F_1^{\fp}$-flat;
		
	%	6) $R$ is $\mathcal F_1^{\fp}$-coherent and any cofinitely generated module is $\mathcal F_1^{\fp}$-flat;
		
		6) Any element of $\mathcal F_1^{\fp}$ is projective.
\end{thm}

\begin{proof} 1) $\Leftrightarrow$ 2) It holds by Proposition \ref{2.3}.\\
	2) $\Leftrightarrow$ 3) It is clear as any module is a submodule of an injective module and, by Proposition \ref{2.2}, $\mathcal F_1^{\fp}\mathcal F(R)$ is stable under submodules.\\
	 2) $\Leftrightarrow$ 4) It is straightforward as free modules are $\mathcal F_1^{\fp}$-flat and any module is a quotient of a free module.\\	
		2) $\Rightarrow$ 5) It is straightforward using Theorem \ref{3.2}.\\
	5) $\Rightarrow$ 2) Let $M$ be an $R$-module. It is known that there exists a set $I$ such that $M$ is isomorphic to a submodule of some direct product $\prod\limits_{i\in I}\E(S_i)$ with the $S_i$ are simple modules. As $R$ is supposed to be right $\mathcal F_1^{\fp}$-coherent, then, by Theorem \ref{3.2}, $\prod\limits_{i\in I}\E(S_i)$ is $\mathcal F_1^{\fp}$-flat. Hence, by Proposition \ref{2.2}, $M$ is $\mathcal F_1^{\fp}$-flat, as desired.\\ 
2) $\Rightarrow$ 6) Assume that $\mathcal F_1^{\fp}\mathcal F(R)$ coincides with the class of all left $R$-modules. Let $\Mod^{\fp}(R)$ designate the class of finitely presented right $R$-modules. Then  $^{\top}\mathcal F_1^{\fp}\mathcal F(R)=\mathcal F(R)$, so that, $$^{\top}\mathcal F_1^{\fp}\mathcal F(R)\cap \Mod^{\fp}(R)=\mathcal F(R)\cap\Mod(R)^{\fp}\subseteq \mathcal P(R).$$ Also, it is easily seen that $\mathcal F_1^{\fp}\subseteq$ $^{\top}\mathcal F_1^{\fp}\mathcal F(R)$. Hence $$\mathcal F_1^{\fp}\subseteq\text { }^{\top}\mathcal F_1^{\fp}\mathcal F(R)\cap \Mod^{\fp}(R)\subseteq \mathcal P(R),$$ as desired.\\ 
	6) $\Rightarrow$ 2) Assume that $\mathcal F_1^{\fp}\subseteq \mathcal P(R)$. Then $\mathcal P(R)^{\top}\subseteq \mathcal F_1^{\fp \top}=\mathcal F_1^{\fp}\mathcal F(R)$. Hence $\mathcal F_1^{\fp}\mathcal F(R)$ coincides with the class of all left $R$-modules completing the proof.\\

\end{proof}

\begin{thm}\label{4.6}
	Let $R$ be a ring. Then the following assertions are equivalent.
	
	1) Any finitely generated $R$-module is $\mathcal F_1$-flat;
	
	2) Any $R$-module is $\mathcal F_1$-flat; 
	
	3) Any injective $R$-module is $\mathcal F_1$-flat;
	
	4) Any quotient of $\mathcal F_1$-flat module is $\mathcal F_1$-flat;

	%	5)  $R$ is $\mathcal F_1$-coherent and any cocyclic module is $\mathcal F_1$-flat;
	
	%	6) $R$ is $\mathcal F_1$-coherent and any cofinitely generated module is $\mathcal F_1$-flat;
	
	5) $\mathcal F_1=\mathcal F(R)$;
	
	6) $\FFD(R)=0$.\\
	Moreover, if $R$ is a commutative Noetherian ring, then the above assertions are equivalent to the following one:
	
	7) $\depth(R_p)=0$ for each prime ideal $p$ of $R$.
	
\end{thm}
\begin{proof} 1) $\Leftrightarrow$ 2) It follows from Proposition \ref{2.3}.\\
	2) $\Leftrightarrow$ 3) It is clear as any module is a submodule of an injective module and, by Proposition \ref{2.2}, $\mathcal F_1\mathcal F(R)$ is stable under submodules.\\
	2) $\Leftrightarrow$ 4) It is straightforward as free modules are $\mathcal F_1$-flat and any module is a quotient of a free module.\\ 
	5) $\Leftrightarrow$ 6) It is direct.\\	
		2) $\Rightarrow$ 5) Assume that any $R$-module is $\mathcal F_1$-flat. Then, as $(\mathcal F_1,\mathcal F_1\mathcal F(R))$ is a torsion theory, we obtain $$\mathcal F_1=\text { }^{\top}\mathcal F_1\mathcal F(R)=\text { }^{\top}\lMod(R)=\mathcal F(R),$$ where $\lMod(R)$ denotes the class of all left $R$-modules, as desired.\\
	6) $\Leftrightarrow$ 7) It follows from \cite[Theorem 2.4]{AB} in the context of commutative Noetherian rings.\\ 
	6) $\Rightarrow$ 2) Assume that $\mathcal F_1=\mathcal F(R)$. Then, as $(\mathcal F_1,\mathcal F_1\mathcal F(R))$ is a Tor-torsion theory, we get $\mathcal F_1\mathcal F(R)=\mathcal F_1^{\top}=\mathcal F(R)^{\top}=\lMod(R)$ completing the proof.
	
\end{proof}

\section{$\mathcal F_1$-flat modules and specific rings}

This section aims at discussing inherent homological properties of a ring $R$ in terms of those of its $\mathcal F_1$-flat modules.\\

We begin by characterizing the homological dimensions of a ring $R$ by the vanishing of the functors Ext and Tor with respect to the class  $\mathcal F_1\mathcal F(R)$ of $\mathcal F_1$-flat modules.

\begin{prop}\label{5.1} Let $R$ be a ring. Let $M$ be a left $R$-module and $n\geq 1$ an integer. Then the following statements are equivalent:

		1) $\pd(M)\leq n$;
		
		2) $\Ext^{n+1}_{R}(M,N)=0$ for each $\mathcal F_1$-flat left $R$-module $N$;
		
		3) $\Ext^{n+1}_{R}(M,N)=0$ for each $\mathcal F_1^{\fp}$-flat left $R$-module $N$.
\end{prop}

The proof requires the following preparatory result.

\begin{lem}\label{5.2}
	Let $R$ be a ring and let $M$ be a left $R$-module. Then the following assertions are equivalent:
	
	1) $M$ is projective.
	
	2) $\Ext_R^1(M,N)=0$ for each $\mathcal F_1$-flat left module $N$;
	
3)	$\Ext^{1}_{R}(M,N)=0$ for each $\mathcal F_1^{\fp}$-flat left $R$-module $N$.
\end{lem}

\begin{proof} 1) $\Rightarrow$ 3) $\Rightarrow$ 2) are clear.\\
	2) $\Rightarrow$ 1) Assume that $\Ext_R^1(M,N)=0$ for each $\mathcal F_1$-flat module $N$. Consider the following short exact sequence of left $R$-modules $ 0\longrightarrow K\longrightarrow P\longrightarrow M\longrightarrow 0 $ with $P$ projective. Then, as $P$ is $\mathcal F_1$-flat, $K$ is $\mathcal F_1$-flat, by Proposition \ref{2.2}. Therefore $\Ext^{1}_{R}(M, K)=0$ and thus the sequence $ 0\longrightarrow K\longrightarrow P\longrightarrow M\longrightarrow 0 $ splits. It follows that M is projective, as desired.
\end{proof}

\begin{proof}[Proof of Proposition \ref{5.1}] 1) $\Rightarrow$ 3) $\Rightarrow$ 2) are clear.\\
	 2) $\Rightarrow$ 1) If $n=0$, then we are done using Lemma \ref{5.2}. Assume that $n\geq 1$. Let $$ \xymatrix{\cdots\ar[r]& P_{1}\ar[r]^{d_{1}}&P_{0}\ar[r]^{\varepsilon} & M\ar[r]& 0} $$ be a projective resolution of $M$ and let $N$ be an $\mathcal F_1$-flat left $R$-module. Let $K_{0}=$ Ker$(\varepsilon)$ and $K_{i}=$ Ker$(d_{i})$ for each integer $i\geq 1$. Then $$\Ext^{1}_{R}(K_{n},N)\cong \Ext^{n+1}_{R}(M, N)=0.$$ Hence, by Lemma \ref{5.2}, $K_{n}$ is projective, and thus $\pd_R(M)\leq n$.
\end{proof} 

\begin{cor}\label{5.3} Let $R$ be a ring. Then 
$$\begin{array}{lll} \gdim(R)&=&\max \{\id_R(M):M\in\mathcal F_1\mathcal F(R)\}\\
&=&\max \{\id_R(M):M\in\mathcal F_1^{\fp}\mathcal F(R)\}.\end{array}$$\end{cor}

\begin{proof}
	First, note that $$\gdim(R)\geq\max \{\id_R(M):M\in\mathcal F_1^{\fp}\mathcal F(R)\} \geq\max \{\id_R(M) : M\in\mathcal F_1\mathcal F(R)\}.$$ If $\max \{\id_R(M):M\in\mathcal F_1\mathcal F(R)\}=+\infty$, then we are done. Assume
	that there exists an integer $n\geq 0$ such that $\id_R(M)\leq n$ for any  $\mathcal F_1$-flat $R$-module $M$. Then $\Ext^{n+1}_R(K,M)=0$ for any $\mathcal F_1$-flat $R$-module $M$ and any $R$-module $K$. Hence, by
	Proposition \ref{5.1}, $\pd_R(K)\leq n$ for any $R$-module $K$. It follows that $\gdim(R) \leq n $ and thus the desired equalities follow.
	\end{proof}

The next two results are immediate consequences of Corollary \ref{5.3}.

\begin{cor}\label{5.4} Let $R$ be a ring. Then the following assertions are equivalent.
	
1)	$R$ is left semi-simple;

2) Any $\mathcal F_1$-flat left $R$-module is injective;

3) Any $\mathcal F_1^{\fp}$-flat left $R$-module is injective.
\end{cor}

\begin{cor} \label{5.5}Let $R$ be a ring. The following assertion are equivalent:
	
	1) $R$ is left hereditary; 
	
	2) $\id_R(M)\leq 1$ for each $\mathcal F_1$-flat left $R$-module $M$.
\end{cor}

This provides another characterization of a Dedekind domains.

\begin{cor} Let $R$ be an integral domain. Then the following assertions are equivalent:
	
1)	$R$ is a Dedekind domain;

2) $\id_R(M)\leq 1$ for each torsion-free $R$-module $M$. 
\end{cor}

\begin{proof}
	Combine Corollary \ref{2.17} and Corollary \ref{5.5}.
\end{proof}

\begin{prop}\label{5.6}
	Let $R$ be a ring and let $M$ be a left $R$-module. Let $n\geq 1$ be an integer. If $\Ext_R^n(N,M)=0$ for each $\mathcal F_1$-flat $R$-module $N$, then $\id_R(M)\leq n$.
\end{prop}

\begin{proof} Assume that $\Ext_R^n(N,M)=0$ for each $\mathcal F_1$-flat $R$-module $N$. Let $K$ be an $R$-module and $0\longrightarrow N\longrightarrow P\longrightarrow K\longrightarrow 0$ be an exact sequence such that $P$ is a projective module. Then $N$ is $\mathcal F_1$-flat and thus $\Ext_R^{n+1}(K,M)\cong\Ext_R^{n}(N,M)=0$. It follows that $\id_R(M)\leq n$, as desired.
\end{proof}

\begin{cor}\label{5.7} Let $R$ be a ring. Then\\ 
	 $$\max\{\pd_R(M):M\in\mathcal F_1\mathcal F(R)\} \leq\gdim(R)\leq 1+\max \{\pd_R(M):M\in\mathcal F_1\mathcal F(R)\}.$$
\end{cor}
 
\begin{proof} It suffices to prove that $\gdim(R)\leq 1+\max \{\pd_R(M):M\in\mathcal F_1\mathcal F(R)\}.$ Effectively, assume that $\max \{\pd_R(M):M\in\mathcal F_1\mathcal F(R)\}\leq n$ for some positive integer $n$. Then Ext$_R^{n+1}(M,N)=0$ for each $\mathcal F_1$-flat $R$-module $M$ and each $R$-module $N$. Hence, by Proposition \ref{5.6}, $\id_R(N)\leq n+1$ for each $R$-module $N$, and thus $\gdim(R)\leq n+1$. It follows that $\gdim(R)\leq 1+\max \{\pd_R(M):M\in\mathcal F_1\mathcal F(R)\}$ completing the proof.
	\end{proof}

\begin{cor}\label{5.8} Let $R$ be any ring. Then the following statements are equivalent:
	
	1) Every $\mathcal F_1$-flat left $R$-module is projective.
	
	2) $R$ is left perfect and left hereditary;
	
	3) $R$ is left perfect and $\wdim(R)\leq 1$;
\end{cor}

\begin{proof}
	2) $\Leftrightarrow$ 3) It is obvious.\\
	3) $\Rightarrow$ 1) Assume that $R$ is left perfect and $\wdim(R)\leq 1$. Let $M$ be an $\mathcal F_1$-flat left $R$-module. Then, as $\mathcal F_1=\Mod(R)$, $M$ is flat. Now, since
	$R$ is left perfect, then $M$ is projective.\\
	1) $\Rightarrow$ 2) Assume that any $\mathcal F_1$-flat left $R$-module is projective. Then, in particular, any flat module is projective, so that $R$ is left perfect. Moreover, by Corollary \ref{5.7}, $R$ is left hereditary completing the proof.
\end{proof}

%\begin{cor}\label{c2.7} Let $R$ be a ring. Then the following assertions are equivalent.
	
%	1) $\gdim(R)<+\infty $;
	
%	2) The set $\{\pd_R(M):M\in\mathcal F_1\mathcal F(R)\}$ is bounded. \end{cor}

\begin{prop} \label{5.9}
	Let $R$ be a ring. Then either $R$ is von Neumann regular or $$\wdim(R)=1+\max\{\fd_R(M):M\in\mathcal F_1\mathcal F(R)\}.$$
\end{prop}

We need the following lemma.

\begin{lem}\label{5.10}
	Let $R$ be a ring and $M$ a right $R$-module. Let $n\geq 1$ be an integer. Then the following assertions are equivalent:
	
	1) $\fd_R(M)\leq n$;
	
	2) $\Tor_n^R(M,N)=0$ for each $\mathcal F_1$-flat $R$-module $N$.
\end{lem}

\begin{proof}	1) $\Rightarrow$ 2) The argument uses induction on $n$. If $n=1$, that is, $M\in\mathcal F_1$, then $\Tor_1^R(M,N)=0$ for each $\mathcal F_1$-flat $R$-module $N$. Assume that $n\geq 2$ and that $\fd_R(M)\leq n$. Let $0\longrightarrow K\lr F\lr M\lr 0$ be an exact sequence such that $F$ is flat. Then $\fd_R(K)\leq n-1$ and thus, by induction, $\Tor_{n-1}^R(K,N)=0$ for each $\mathcal F_1$-flat $R$-module $N$. Hence $\Tor_n^R(M,N)\cong\Tor_{n-1}^R(K,N)=0$ for each $\mathcal F_1$-flat $R$-module $N$, as desired.  \\
	2) $\Rightarrow$ 1) Assume that $\Tor_n^R(M,N)=0$ for each $\mathcal F_1$-flat $R$-module $N$. Let $K$ be an $R$-module and let $0\longrightarrow N\longrightarrow F\longrightarrow K\longrightarrow 0$ be an exact sequence such that $F$ is a flat module. Then, by Proposition \ref{2.2}, $N$ is $\mathcal F_1$-flat. Hence $\Tor_{n+1}^R(M,K)\cong\Tor_n^R(M,N)=0$. It follows that $\fd_R(M)\leq n$, as desired.
\end{proof}

\begin{proof}[Proof of Proposition \ref{5.9}] Assume that $\wdim(R)\geq 1$. Let $n\geq 1$ be an integer. Then $\wdim(R)\leq n$ if and only if $\fd_R(M)\leq n$ for any right $R$-module $M$ if and only if $\Tor_n^R(M,N)=0$ for each right $R$-module $M$ and each $\mathcal F_1$-flat $R$-module $N$ (by Lemma \ref{5.10}) if and only if $\fd_R(N)\leq n-1$ for each $\mathcal F_1$-flat $R$-module $N$ if and only if $1+\max\{\fd_R(N):N\in\mathcal F_1\mathcal F(R)\}\leq n$. The desired equality then follows easily.
		\end{proof}

\end{document}